\documentclass[pre,reprint,aps,flushbottom,showpacs]{revtex4-1}

\usepackage{amssymb,amsmath,amsxtra}
\usepackage{latexsym}

\usepackage[dvips]{graphicx}            

\usepackage{epsfig}
\usepackage{psfrag} 

\newenvironment{proof}{ \indent\textbf{Proof}.}{\hfill $\Box$}

\newtheorem{theorem}{Theorem}[section]

\newtheorem{lemma}[theorem]{Lemma}

\numberwithin{equation}{section}

\def\beq{\begin{equation}}
\def\eeq{\end{equation}}
\def\beqnn{\begin{eqnarray*}}
\def\eeqnn{\end{eqnarray*}}
\def\bea{\begin{eqnarray}}
\def\eea{\end{eqnarray}}
\def\figsize{.4\textwidth}

\newcommand{\eps}{\epsilon}
\newcommand{\R}{\mathbb{R}}

\begin{document}

\title{Fourier methods for the perturbed harmonic oscillator in linear and nonlinear Schr\"odinger equations}

\author{Philipp \surname{Bader}
}
\email{phiba@imm.upv.es}
\author{Sergio \surname{Blanes}
}
\email{serblaza@imm.upv.es}
\affiliation{   Universitat Polit\`ecnica de Val\`encia,
                Instituto de Matem\'atica Multidisciplinar,
                E-46022 Valencia, Spain}


\begin{abstract}
We consider the numerical integration of the Gross-Pitaevskii
equation with a potential trap given by a time-dependent harmonic
potential or a small perturbation thereof. Splitting methods are
frequently used with Fourier techniques since the system can be
split into the kinetic and remaining part, and each part can be
solved efficiently using Fast Fourier Transforms. To split the
system into the {quantum harmonic oscillator problem}
and the remaining part allows to get higher accuracies in many
cases, but it requires to change between Hermite basis functions
and the coordinate space, and this is not efficient for time-dependent
frequencies or strong nonlinearities.
We show how to build new methods which {combine} the advantages of
using Fourier methods {while solving the time-dependent harmonic
oscillator exactly (or with a high accuracy by using a Magnus
integrator and an appropriate decomposition)}.
\end{abstract}



\maketitle

\section{Introduction}  \label{sec.1}

The numerical integration of the Gross-Pitaevskii equation (GPE)
\[
     i \frac{\partial}{\partial t} \psi (x,t) = \left(%
                                                    -\frac{1}{2\mu} \Delta + V(x,t) +   \sigma(t) |\psi (x,t)|^2%
                                                \right) \psi(x,t)
\]
$x\in \R^d$, describing the ground state of interacting bosons at
zero temperature, the Bose-Einstein condensates, has attracted
great interest \cite{bao03jcp,perez03amc,thalhammer09jcp} after
the first experimental realizations \cite{bec-experiment}. We
present a new efficient way to solve a special class of GPE,
namely that of weakly interacting bosons in a single
time-dependent trap. To be more specific, the potential trap $V$
is taken to be a perturbation of the (time-dependent)
$d$-dimensional harmonic oscillator, i.e. $V(x,t)=x^TM(t)x + \eps
V_I(x,t)$ where $M(t)\in \R^{d\times d}$ is a positive definite
matrix and $\eps V_I(x,t)$ is a small perturbation. The real
scalar function $\sigma$ originates from the mean-field
interaction between the particles and corresponds to repulsive or
attractive forces for positive or negative values of $\sigma(t)$,
respectively \cite{pet01}.
Notice that the non-interacting case, $\sigma\equiv0$, corresponds to the linear Schr\"odinger equation.

Several methods have been analyzed to compute both the time evolution and the ground state of the
GPE in the course of the last decade \cite{bao03jcp,bao05jsc,dion03pre,perez03amc,thalhammer09jcp},
among them finite differences, Galerkin spectral methods and pseudospectral methods for Fourier or Hermite basis expansions.
It has been concluded \cite{perez03amc} that these pseudospectral methods perform best
for a wide parameter range for the GPE.
The Fourier type methods can be implemented with Fast Fourier Transform (FFT) algorithms
since the trapping $V$ causes the wave function to vanish asymptotically thus allowing
to consider the problem as periodic on a sufficiently large spatial interval.
Their advantages are high accuracy with a moderate number of mesh points and low computational cost.
For harmonic oscillator (HO) problems, however, the exact solution is known
and by expanding the solution in Hermite polynomials, highly accurate results
are obtained if the HO is solved separately \cite{bao05jsc,dion03pre,perez03amc,thalhammer09jcp}.

It is claimed \cite{perez03amc}, that either Hermite or Fourier pseudospectral methods are the most efficient,
the choice depending on the particular parameter set.
Motivated by these results, we show how both methods are combined to retain both the accuracy
of the Hermite method and the speed of the Fourier transforms, i.e. to rewrite
the Hermite method as a single simple pseudospectral Fourier scheme.
We have found, that this approximation performs, for the studied problem class,
always equal to or better than the original Fourier method and therefore has to compete
with Hermite expansions only.
Hermite schemes suffer from large computational costs when the number of basis terms in the expansions
is altered along the integration or taken very large,
being the case for time dependent trap frequencies $M(t)$ or strong nonlinearities $\sigma(t)$.
It is in this setup, where our new method substantially improves the Hermite performance,
and it can indeed be regarded as the optimal choice for the number of Hermite basis functions (for an equidistant grid)
at each time step.

For the ease of notation, we restrict ourselves to the one-dimensional problem
\beq  \label{GP1}
     i \frac{\partial}{\partial t} \psi = H_0(t) \psi +
                    \left( \varepsilon V_I(x,t) + \sigma(t) |\psi|^2 \right) \psi
\eeq where
\beq  \label{GP1b}
     H_0(t) = \frac{1}{2\mu} p^2 +  \frac{1}{2}\mu\omega^2(t) x^2
\eeq
and $p=-i\frac{\partial}{\partial x}$.
The boundary conditions imposed by the trap require the wave function
to go to zero at infinity, and up to any desired accuracy, we can assume $\psi (x,t)$
and all its derivatives to vanish outside a finite region, say $[a,b]$,
which we divide using a mesh (usually with $N=2^k$ points to allow a simple use of the FFT
algorithms). Then, the partial differential equation (\ref{GP1})
transforms into a system of ordinary differential equations (ODEs)
\beq\label{GP1ODE}
    i \frac{d}{d t} u (t) = \hat H_0(t) u(t) +
                            \left(\varepsilon V_I(X,t) + \sigma(t) |u(t)|^2 \right) u(t),
\eeq
and the harmonic part becomes
\beq\label{GPODE1b}
    \hat H_0(t) = T + \frac{1}{2}\mu\omega^2(t) X^2,
\eeq with $u\in\mathbb{C}^N$, where $u_i(t)\simeq \psi (x_i,t)$,
$x_i=a+ih, \quad i=0,1,\ldots, N-1, \ h=(b-a)/N$,
$X=diag\{x_0,\ldots,x_{N-1}\}$ and $T$ denotes a discretization of
the kinetic part.

This system of nonlinear ODEs can be numerically solved by
standard all purpose ODE methods. However, because of the
particular structure of this problem, different numerical methods
can differ considerably in accuracy as well as
in computational cost and stability. In addition, the structural properties
of the system lead to the existence of several preserved quantities like the
norm and energy (for the autonomous case).

The accurate preservation of these quantities as well as the error propagation and performance
of splitting methods explain why they are frequently recommended for the time integration
\cite{bao05jsc,perez03amc,thalhammer09jcp} and make them subject of investigation in this work.

\section{Splitting methods}

Let us consider the separable system of ODEs \beq\label{aut_eq}
        u'=A(u)+B(u), \qquad u(t_0)=u_0 \in \mathbb C^N,
\eeq
where we assume that both systems
\beq\label{equationsAB}
        u'=A(u), \qquad u'=B(u)
\eeq
can either be solved in closed form or accurately integrated. If
$\varphi^{[A]}_t$, $\varphi^{[B]}_t$ represent the exact flows
associated to (\ref{equationsAB}) then, to advance the solution
one time step, $h$, we can use, for example, the composition
$\psi_h^{[1]}= \varphi^{[A]}_{h} \circ \varphi^{[B]}_{h}$ (i.e.
$u(t_0+h)\simeq \psi_h^{[1]}(u_0)=\varphi^{[A]}_{h}\Big(
\varphi^{[B]}_{h}(u_0)\Big)$),
which is known as the first-order Lie-Trotter method.
A method has order $p$ if $\psi_h^{[p]}=\varphi_h+\mathcal{O}(h^{p+1})$ where $\varphi_t$
denotes the exact global flow of \eqref{aut_eq}. Sequential application of the two
first-order methods $\psi_h^{[1]}$ and its adjoint $\psi_h^{[1]*}=\varphi^{[B]}_{h} \circ \varphi^{[A]}_{h}$
with half time step yields the second order time-symmetric methods
\begin{align}
\psi_{h,A}^{[2]} &=  \varphi^{[A]}_{h/2} \circ \varphi^{[B]}_{h}
   \circ \varphi^{[A]}_{h/2}
\label{LeapFrogA}
\\
\psi_{h,B}^{[2]} &=  \varphi^{[B]}_{h/2} \circ \varphi^{[A]}_{h}
   \circ \varphi^{[B]}_{h/2}
\label{LeapFrogB}
\end{align}
(referred as $ABA$ and $BAB$ compositions). The contraction via the (1-parameter-)group
property of the flows that eliminated one computation is called First Same As Last (FSAL) property
and can also be used with higher order $m-$stage compositions
\begin{equation}\label{composition}
\psi_h^{[p]}=\varphi^{[A]}_{a_{m}h} \circ \varphi^{[B]}_{b_{m}h}
 \circ \dots  \circ
 \varphi^{[A]}_{a_1h} \circ \varphi^{[B]}_{b_1h}
\end{equation}
if $a_m=0$ or $b_m=0$ and repeated application of the scheme without requiring output.
For linear problems, it is usual to replace the flow-maps by exponentials (for nonlinear problems
the same is possible using exponentials of Lie operators).
In this notation, the equation $iu'=A(u)+B(u)$,
whose formal solution for the evolution operator is denoted by $\phi^{[A+B]}_t=e^{-it(A+B)}$,
is approximated for one time step, $h$, by the order $p$ composition  \eqref{composition} or,
equivalently, 
\begin{equation}  \label{prod-gen}
  \psi_h^{[p]} \equiv  e^{-ih a_m A} \, e^{-ih b_m B} \, \cdots \,
   e^{-ih a_1 A}  \, e^{-ih b_1 B}.
\end{equation}

We keep in mind that, in a nonlinear problem, if $B$ depends on
$u$, it has to be updated at each stage because $u$ changes during
the evolution of  $e^{-ih a_i A}$.

Backward error analysis shows that the action of a splitting
method is equivalent to solve exactly, for one time step $h$, a
perturbed differential equation (for sufficiently small $h$) that
can contain higher order commutators
\begin{align}\label{eq:BEA}
  u'     &= \left(A+B + h \alpha_{1,1} [A,B] \right.  \\
         &  \left.+ h^2( \alpha_{2,1}[A,[A,B]]+
                \alpha_{2,2}[B,[B,A]]) + \ldots \right) u   \nonumber
\end{align}
($[A,B]:=AB-BA$, and for simplicity, we have denoted $\
A=A(u)\cdot \nabla, \ B=B(u)\cdot \nabla$) where the coefficients
$a_i, b_i$ are chosen in order to cancel out the coefficients
$\alpha_{i,j}$ up to a given order.
It is thus important to analyze the dominant contributions to the
error from the commutators, in order to choose the most
appropriate method or to build new ones.

There exist many different splitting methods which are designed
for different purposes, depending on the structure of the problem,
the {desired} order, the {required} stability, etc.
\cite{blanes08sac,blanes02psp,hairer06gni,mclachlan95cmi,mclachlan02sm,yoshida90coh,suzuki90fdo}.
If the operator $A$ corresponds to the kinetic part, which is
quadratic in momenta, and the operator $B$ is the potential,
diagonal in coordinate space, the commutator $[B,[B,[B,A]]]$
vanishes and partitioned Runge-Kutta-Nystr\"om methods become
favorable. The coefficients for this family of methods have to
solve a significantly reduced number of order conditions and, in
general, their performance is superior to splitting methods
designed for general separable problems, in addition to being more
stable. Efficient schemes of order 4 and 6 are obtained in
\cite{blanes02psp}. On the other hand, when $H_0$ is the dominant
part, it is worth to take a closer look at the split \eqref{GP1}.
This would correspond to $\|B\|\ll \|A\|$ and for this case, when
facing autonomous problems, there exist tailored methods which
have shown a high performance in practice. {Writing} the equation
as $ iu' = (A + \epsilon B) u$ (with $\epsilon$ a small
parameter), it is clear that the local error of the second order
methods (\ref{LeapFrogA}) or (\ref{LeapFrogB}) comes from the
commutators at third order ($[A,[A,\epsilon B]]$ and $[\epsilon
B,[A,\epsilon B]]$) and we can say that the local error is of
order $\mathcal{O}(\epsilon h^3+\epsilon^2 h^3)$. The coefficients
$a_i,b_i$ in the general composition (\ref{prod-gen}) can be
chosen to cancel the dominant error terms, say, the
$\mathcal{O}(\epsilon h^r)$ terms for relatively large values of
$r$. Then, one can denote the {effective} order of a method by
$(r,p)$ with $r\geq p$ when the local error is given by
$\mathcal{O}(\epsilon h^{r+1}+\epsilon^2 h^{p+1})$. The method is
of order $p$, but in the limit $\epsilon\rightarrow 0$ it is
considered as of order $r\geq p$. Using this split allows to gain
a factor $\epsilon$ in the accuracy even for general splitting
methods where $r=p$. In \cite{mclachlan95cmi}, several methods of
order $(r,2)$ for $r\leq10$ are obtained with all coefficients
$a_i,b_i$ being positive and some other schemes of order $(r,4)$
for $r=6,8$ are {presented}. For near-integrable systems, these
last methods are the most efficient and stable ones. Despite the
gain of accuracy, the split into a dominant part and a small
perturbation is left unconsidered when it leads to involved or
computationally costly algorithms. This issue is addressed in this
work.

To take account for the time dependence in the vector fields in \eqref{GP1},
a more detailed analysis is required.
The general separable equation
\beq\label{nonaut_eq}
    iu'=A(u,t)+\epsilon B(u,t)
\eeq
can be solved by considering the time as two new independent coordinates
\beq\label{equationsAB2}
    \left\{ \begin{array}{rcl}
                iu'  & = & A(u,t_2) \\
                t_1' & = & 1
            \end{array} \right. ,
    \qquad
    \left\{ \begin{array}{rcl}
                iu'  & = & \epsilon B(u,t_1) \\
                t_2' & = & 1
            \end{array} \right. .
\eeq

This standard split is equivalent to a new system in an extended
phase space which is not near-integrable and the highly efficient
and stable methods designed for those problems lose their
excellent performance. Following \cite{blanes10sac}, the near
integrability is recovered if we introduce only one variable as
follows \beq\label{equationsAB3}
    \left\{ \begin{array}{rcl}
                iu'  & = & A(u,t_1) \\
                t_1' & = & 1
            \end{array} \right. ,       \qquad      iu'  =  \epsilon B(u,t_1) .
\eeq
This split requires to solve exactly the equation $iu'=A(u,t)$ for the
corresponding fractional time steps and to freeze the time when
solving $iu'=\epsilon B(u,t_1)$ (see Table~\ref{algorithm}).

\begin{table}[h!]
\caption{Algorithm for the numerical integration of the system
(\ref{equationsAB3}) by the composition (\ref{composition}).}
\label{algorithm} {
 \small
\begin{tabular}{c}
\begin{tabular}{l}
\hline
  {\bf Algorithm for one time step $t_n\rightarrow t_n+h$}
 \\
\hline
 $
\begin{array}{l}
  t^{[0]}=t_n \\
  {\bf do} \ \ i=1,m \\
   \quad  solve: \ u'=B(u,t^{[i-1]}), \quad \quad t\in[t^{[i-1]},t^{[i-1]}+b_ih] \\
   \quad  solve: \ u'=A(u,t), \qquad \qquad t\in[t^{[i-1]},t^{[i-1]}+a_ih] \\
   \quad  \qquad \ \quad t^{[i]}=t^{[i-1]}+a_ih \\
  {\bf enddo}  \\
  t_{n+1}=t^{[m]} \\
\end{array}  $
 \\
 \hline
\end{tabular}
\end{tabular}
}
\end{table}

We show that the exact solution of the non-autonomous problem, in
our setting the dominant part $H_0$ of \eqref{GP1}, \beq
\label{HO(t)}
     i \frac{\partial}{\partial t} \psi =
            \left(  \frac{1}{2\mu} p^2 +
                    \frac{1}{2}\mu\omega^2(t) x^2  \right) \psi
\eeq
is easily computed for a time step using Fourier
transforms. Before giving the details on the time integration,
some remarks on the formal solution are necessary.

It is well known that $H_0$ is an element of the Lie algebra spanned by the operators
$\{E=x^2/2, F=p^2/2, G=\frac12(px+xp)\}$, where $\mu=1$ for simplicity, and its commutators are
\[
    [E,F]=iG, \quad  [E,G]=2iE, \quad    [F,G]=-2iF.
\]
This is a three-dimensional Lie algebra and the solution,
$\psi(x,t)=U(t,0)\psi(x,0)$, of (\ref{HO(t)}) can be expressed as a
single exponential using the Magnus series expansion
\cite{magnus54ote,blanes09tme} 
or as a product of exponentials \cite{wei63las}. It is possible to formulate the
evolution operator $U(t,0)$ in many different ways, the most appropriate depending on
the particular purpose, e.g. using the Magnus expansion
\beq\label{HO(t)sol}
U(t,0) = \exp\left( f_1(t) E + f_2(t) F + f_3(t) G \right)
\eeq
for certain functions $f_i(t)$ \cite{remarkPhilipp1}.
Approximations of \eqref{HO(t)sol} for one time step, $h$, on the
other hand, are easily obtained, e.g. a fourth-order
commutator-free method is given by \cite{blanes00smf}
\begin{align}\label{CommutatorFree}
    U(t+h,t) & = \exp\left( \frac{h}{2}(\frac12p^2+\omega_L^2\frac12x^2)
                    \right) \\
             & \times \exp\left( \frac{h}{2}(\frac12p^2+\omega_R^2\frac12x^2)  \right) +
                    \mathcal{O}(h^5) \nonumber
\end{align}
where
\[
  \omega_L^2 = \alpha \omega_1^2 + \beta \omega_2^2, \qquad
  \omega_R^2 = \beta \omega_1^2 + \alpha \omega_2^2
\]
with $\omega_i=\omega(t_n+c_ih), \ c_1=\frac12-\frac{\sqrt{3}}{6},
\ c_2=\frac12+\frac{\sqrt{3}}{6}$, and
$\alpha=\frac12-\frac{1}{\sqrt{3}}, \ \beta=1-\alpha$. It can be
considered as the composition of the evolution for half time step
of two oscillators with averaged frequencies, using the
fourth-order Gauss-Legendre quadrature rule to evaluate
$\omega(t)$. Different quadrature rules can also be be used and
correspond to different averages along the time step, see
\cite{blanes00smf,blanes09tme}. In the limit when $\omega$ is
constant the exact solution is recovered. Higher order
approximations are available, if more accurate results are
desired, by approximating the functions $f_i$ in \eqref{HO(t)sol}
via truncated Magnus expansions.

Our objective is to obtain a factorization of the solution which only
involves terms proportional to $E$ or $F$ since they are easy to
compute as we show in the following paragraph.

Starting from \eqref{CommutatorFree} or high order approximates of \eqref{HO(t)sol},
the main result of this work is the natural decomposition for
the application of Fourier spectral methods.

\subsection{Hamiltonians for spectral methods}
We now analyze how to compute the evolution of different parts of
the Hamiltonian by spectral methods.
The spatial derivative (or kinetic part) associated to the
semidiscretized problem (\ref{GP1ODE}) can be solved in the
momentum space by noting that
\begin{equation}\label{GP1ODE_T}
 i \frac{d}{d t} u (t) = T  u(t)
 = \mathcal{F}_N^{-1} D_N \mathcal{F}_N \, u ,
\end{equation}
where $D_N$ is diagonal and $\mathcal{F}_N$ denotes the discrete
Fourier transform of length $N$, whose computation  can be
accomplished by the FFT algorithm with $\mathcal{O}(N \log N)$
floating point operations. The solution of (\ref{GP1ODE_T}) for
one time step, $h$, is given by
\[
 u(t+h) = \mathcal{F}_N^{-1} e^{-ihD_N} \mathcal{F}_N \, u(t)
\]
which requires two FFT calls. The exponentials in $e^{-ihD_N}$ need to
be computed {only} once and can be reused at each step such that the
cost of the action of $e^{-ihD_N}$ corresponds to $N$ complex products.

For the remaining part, the following well-known result is very
useful
\begin{lemma}\label{lemma1}
If $F$ is real valued, the equation
\beq  \label{NLS-pot}
    i \frac{\partial}{\partial t} \phi (x,t) = F(x,|\phi (x,t)|) \phi (x,t),
\eeq
leaves the norm invariant, $|\phi(x,t)|=|\phi(x,0)|$, and then
\begin{equation} \label{solNLS-pot}
  \phi(x,t) = e^{-itF(x,|\phi(x,0)|)} \phi(x,0).
\end{equation}
\end{lemma}

%
\vspace*{.3cm}

On the other hand, it is well known that the solutions of the
linear Schr\"odinger equation with the harmonic potential 
\beq \label{Harmonic}
    i \frac{\partial}{\partial t} \phi (x,t) =  \frac12 (p^2 + x^2)  \phi (x,t)
\eeq can be expressed in terms of Hermite polynomials:
\beq\label{Hermite}
    \phi (x,t) = \sum_{n=0}^{\infty} c_n e^{-iE_n t} h_n(x)
\eeq
where
\beq\label{EigenHarm}
    E_n=n+\frac12, \qquad h_n(x)=\frac{1}{\pi^{1/4}\sqrt{2^nn!}} H_n(x) e^{-x^2/2}
\eeq
and $H_n(x)$ are the Hermite polynomials satisfying the recursion
\[
    H_{k+1}(x) = 2xH_k(x)-2kH_{k-1}, \qquad k=1,2,\ldots
\]
with $H_0(x)=1, \ H_1(x)=2x$. The weights $c_n$ can be computed
from the initial conditions, $c_n=\int h_n(x) \phi(x,0)\,dx$.

The previous results show how to compute individual parts of the equation
and thus permit different ways of splitting the system in two solvable parts
\beq\label{SepAB}
    i\psi_t = (A + B) \psi
\eeq
We consider the following cases:

{\bf (i) Fourier(F)-split}.

\beq \label{FourierSplit}
    A= \frac{1}{2}\,p^2 ,   \quad   B(t) = \frac{\omega(t)}{2}\,x^2 + \varepsilon V_I(x,t) + \sigma(t) |\psi |^2,
\eeq
 We take the time as a new coordinate, as in (\ref{equationsAB3}),
and evolve it with $A$, which is now autonomous and exactly
solvable, and freeze the time in $B$, which is then solved using
the result from Lemma~\ref{lemma1}. Here, $A$ and $B$ are
diagonal in the momentum and coordinate spaces, respectively, and
we can change between them using the Fourier Transforms.

{\bf (ii) Harmonic oscillator(HO)-split}.
Let for a moment $w=1$, the Hermite expansion then suggests a split
\beq \label{HermiteSplit-const}
    A= \frac1{2} (p^2 + x^2),   \qquad B(t) = \varepsilon  V_I(x,t) + \sigma(t) |\psi |^2,
\eeq
where the solution for the equation $iu'=Au$ can be approximated
using a finite number of Hermite basis functions, i.e.
\begin{equation}\tag{\ref{HermiteSplit-const}b}\label{HermiteM}
    \phi_M (x,t) = \sum_{n=0}^{M-1} c_n e^{-iE_n t} h_n(x).
\end{equation}
Since $B$ is diagonal in the coordinate space it will act as a simple multiplication
if we choose a number of Hermite basis functions and evaluate them at the points of a
chosen mesh (e.g. using the Gauss-Hermite quadrature \cite{lubich08fqt} or on equidistant grid points
\cite{perez03amc}). This split can be of interest if all
contributions from $B(t)$ are small with respect to $A$ and
methods for near-integrable systems are used.
If the frequency is time dependent, the corresponding split is
 \beq \label{HermiteSplit}
  A(t)= \frac1{2} (p^2 + \omega^2(t)x^2),  \qquad B(t) = \varepsilon V_I(x,t) + \sigma(t) |\psi |^2.
 \eeq
In this case, it is convenient to take the time, $t$ as a new
coordinate as shown in (\ref{equationsAB3}). The solution of
$iu'=A(t)u$, in the algorithm in Table~\ref{algorithm}, can be
approximated by the Magnus expansion, e.g. (\ref{HO(t)sol})
or (\ref{CommutatorFree}),
but these factorizations are not appropriate for use with spectral
methods since it would require two sets of basis functions and
additional transformations.

\vspace*{.3cm}

In general, the split {\bf (i)} can be considered faster and
simpler since $A\equiv T$ can be computed in the momentum space,
and one can easily and efficiently change from momentum to
coordinate space via FFTs. The choice {\bf (ii)}, on the other
hand, allows us to take advantage of the structure of a
near-integrable system if, roughly speaking, $\|B\|<\|A\|$, but it
requires to solve the equation for the (time-dependent) harmonic
potential exactly (or with high accuracy). The evolution of the
constant oscillator is easily computed using Hermite polynomials
(see \cite{perez03amc,thalhammer09jcp,lubich08fqt}). As mentioned,
methods tailored for this family of problems show a high
performance. The main drawback, however, is that the evolution for
$e^{-ih a_i A}$ has to be carried out using a basis of Hermite
polynomials whereas the $B$ part is advanced in space coordinates
\cite{lubich08fqt} resulting in possibly costly basis
transformations.

\subsection{Solving the Harmonic oscillator by Fourier methods}
We propose a new method which combines the advantages of both splittings.
It retains the advantages of the HO-split {\bf (ii)} while being as fast to compute
as the F-split in {\bf (i)}. For this purpose, we briefly review some basic concepts of Lie
algebras.

Given $X,Y$ two elements of a given Lie algebra it is well known that
\[
    e^{X} Y e^{-X} =    
                        Y + [X,Y] + \frac12 [X,[X,Y]]+ \ldots
\]
If $\{X_1,\ldots,X_k\}$ are the generators of a finite dimensional Lie algebra and
\[
    Z=\sum_{n=1}^k \alpha_n X_n, \qquad     G=\sum_{n=1}^k a_n X_n
\]
then
\beq\label{ExpAd}
    G(t) = e^{tZ} G e^{-tZ} = \sum_{n=1}^k a_n(t) X_n
\eeq
 where $G(t)$ satisfies the differential equation
\[
  G'(t) = [Z,G(t)], \qquad G(0)=G.
\]
Notice that if $G=x$ or $G=p$ and $Z=H(x,p)$ these equations are
equivalent to the equation of motion for the classical Hamiltonian
$H(x,p)$. It is also well-known that two operators are identical
on a sufficiently small time interval
if they satisfy the same first order differential equation with
the same initial conditions \cite{wilcox67jmp}.

Given $X(x),P(p),W=\frac12(p^2+x^2)$ and an analytic function $F(x,p)$,
we are interested in the following adjoint actions
\begin{subequations}
\begin{align}\label{Evol1}
    e^{-itX} F(x,p) e^{itX} & = F(x,p+tX')  \\
    e^{-itP} F(x,p) e^{itP} & = F(x-tP',p)  \\
    e^{-itW} F(x,p) e^{itW} & = F(\hat x,\hat p)
\end{align}
\end{subequations}
where $X'=dX/dx, \, P'=dP/dp$ and
\begin{align*}
    \hat x & =  \cos(t)x + \sin(t) p  \\
    \hat p & = -\sin(t)x + \cos(t) p.
\end{align*}
In classical mechanics, this corresponds to a kick, a drift and the harmonic oscillator rotation.

As we have seen, the exponentials
can be easily computed by Fourier spectral methods. It is then natural
to ask the question if it is possible to write the solution of
(\ref{HO(t)}) as a product of exponentials which are solvable by
spectral methods. The answer is positive and it is formulated as
follows.

Let us first consider the pure harmonic oscillator, whose result was
obtained in  \cite{chin05foa},
and for which we
present a new proof that applies equally to the general case.
\begin{lemma}\label{Lemma2}
Let $A_1=\frac12 p^2, \ B_1=\frac12 x^2$ and
\beq \label{f(h)-g(h)}
    g(t) = \sin\left(t\right) , \qquad \displaystyle f(t)=\tan\left(t/\,2\right).
\eeq
Then, the following property is satisfied for $|t|<\pi$:
\begin{align}
    e^{-it(A_1+B_1)}&= e^{-if(t)A_1}\ e^{-i g(t) B_1}\ e^{-if(t)A_1}   \label{ABA}   \\
                    &= e^{-if(t)B_1}\ e^{-i g(t) A_1}\ e^{-if(t)B_1}   \label{BAB}
\end{align}
\end{lemma}
\begin{proof}
It suffices to impose that the right hand side of \eqref{ABA} (or
\eqref{BAB}) satisfies the equation (\ref{Harmonic}) with the
identity as initial condition what can be verified by differentiating
and simple applications of the rules (\ref{Evol1},b).

A more constructive way to derive the functions $f,g$ makes use
of the parallelism with the one-dimensional classical harmonic oscillator
with Hamiltonian $H= \frac12 p^2+\frac12 q^2$ and Hamilton equations
\begin{equation}  \label{split_AB}
 \frac{d}{dt} \left\{
 \begin{array}{c}
 q \\
 p  \end{array} \right\} =  \left(
 \begin{array}{ccc}
  0   & \, & 1  \\
  -1 & \, &  0 \end{array} \right)   \left\{
 \begin{array}{c}
  q  \\
  p  \end{array} \right\} = ( A+ B) \left\{
 \begin{array}{c}
  q \\
  p \end{array} \right\}
\end{equation}
where
\begin{equation}    \label{eq.4a}
 A \equiv \left(  \begin{array}{ccc}
               0   & \, &  1  \\
              0   & \, &  0  \end{array} \right),
              \qquad\qquad
  B \equiv \left(  \begin{array}{ccc}
               0   & \, & 0  \\
              -1    & \, & 0  \end{array} \right).
\end{equation}
The Lie algebra generated by the matrices $A,B$ is the same as the
Lie algebra associated to the operators $A_1,B_1$ for the
Schr\"odinger equation with the harmonic potential \eqref{Harmonic}.

The exact evolution operator of (\ref{split_AB})  is
\begin{equation}  \label{exact-Sch}
   O(t) = \left(
 \begin{array}{rcr}
   \cos(t)   &  & \sin(t)  \\
  -\sin(t)  &   & \cos(t)  \end{array}
  \right),
\end{equation}
which is an orthogonal and symplectic $2 \times 2$ matrix.
For the splitted parts, the solutions are easily computed to
\begin{equation} \nonumber\label{eA_eB}
  e^{f(t)A} = \left(
 \begin{array}{ccc}
 1   & \,  & f(t)  \\
  0  & \,  &  1   \end{array} \right), \qquad
   e^{g(t)B} = \left(
 \begin{array}{ccc}
   1   & \, & 0 \\
  -g(t)  &  \, &1  \end{array} \right)
\end{equation}
and then, equating the symmetric composition
\begin{equation} \nonumber\label{efA_egB_efA}
 e^{fA} e^{gB}   e^{fA}  = \left(
 \begin{array}{ccc}
 1 - f\cdot g   & \,  & 2f- f^2\cdot g  \\
  -g  & \,  &  1 - f\cdot g  \end{array} \right),
\end{equation}
to (\ref{exact-Sch}), we obtain (\ref{f(h)-g(h)}) which is valid for
$|h|\leq \pi$. 
The decomposition \eqref{BAB} is derived analogously.
Using the Baker-Campbell-Haussdorf-formula, it is clear, that
both results remain valid, up to the first singularity at $t=\pm\pi$,
when replacing the matrices $A,B$ by the corresponding
linear operators $A_1,B_1$, since all computations are done in
identical Lie algebras.
\end{proof}

\vspace*{.3cm}

\vspace*{.3cm}

It is immediate to generalize this result to the equation
\beq  \nonumber\label{SchrHarmOsc2}
 i \frac{\partial}{\partial t} \psi (x,t) = \left(
 \frac{1}{2\mu} p^2 +
     \mu\frac{\omega^2}{2} x^2 \right) \psi (x,t),
\eeq
$\mu,\omega>0$ {by replacing} (\ref{f(h)-g(h)}) {with}
\[
 g = \frac{1}{\mu \omega}\sin(\omega t), \qquad
 f = \mu \omega\tan\left(\frac{\omega}{2}t\right).
\]
This result is valid for $|t|< t^*\equiv \pi/\omega$.

The following theorems extend this idea to decompositions of
operators that appear after the approximation of the time
dependent parts via \eqref{HO(t)sol} or by the composition \eqref{CommutatorFree}.

\begin{theorem}\label{theorem1}
Let $\alpha,\beta ,\gamma$ be constants, $\eta=\sqrt{\alpha\gamma-\beta^2}$ and
\begin{align}\label{sol:theorem1}
    g(t) & = \gamma/\eta \cdot \sin(\eta t) ,  \\
    f(t) & = \frac{1}{g(t)}\left(1-\cos(\eta t)+\frac{\beta}{\eta}\sin(\eta t)\right), \nonumber  \\
    e(t) & = \frac{1}{g(t)}\left(1-\cos(\eta t)-\frac{\beta}{\eta}\sin(\eta  t)\right)
    .\nonumber
\end{align}
Then, the following decomposition holds for $0\leq t < \pi/\eta$:
\begin{align}
  &\, e^{-i\frac{t}{2}\left(\alpha x^2+\beta (xp+px)+\gamma p^2\right)} \nonumber \\
 =&\, e^{-if(t)\frac{1}{2}x^2}\ e^{-i g(t) \frac{1}{2}p^2}\ e^{-ie(t)\frac{1}{2}x^2}.
   \label{BAB-th1}
\end{align}
\end{theorem}

\begin{proof}
 The proof follows the lines of the proof of Lemma \ref{Lemma2}. The evolution operator
 associated to the classical Hamiltonian $H=\frac{1}{2}\left(\alpha x^2+2\beta xp+\gamma p^2\right)$
 is given by
\beq\label{}
\left( \begin{array}{cc}
   \cos(\eta t)+\frac{\beta}{\eta}\sin(\eta t) &  \frac{\gamma}{\eta}\sin(\eta t)\\
   -  \frac{\alpha}{\eta}\sin(\eta t) & \cos(\eta t)-\frac{\beta}{\eta}\sin(\eta t)
 \end{array} \right).
\end{equation}
and equality to the right hand side of \eqref{BAB-th1}
is verified by straight-forward computation of the matrix exponentials.
The solution is valid until the first singularity at $t=\pi/\eta$.
Using (\ref{Evol1},b), it can be checked that both sides
of \eqref{BAB-th1} satisfy the same differential equation and initial conditions.
\end{proof}

\begin{theorem}\label{theorem2}
Let $\omega_k\in\mathbb{R}$, $c_k=\cos(\omega_kt/2), \
s_k=\sin(\omega_kt/2)$ for $k=L,R$ and
\begin{align}
        \label{sol:theorem2}
    g(t) & = s_Lc_R/\omega_L+c_Ls_R/\omega_R,  \\
        \nonumber
    f(t) & = \frac{1}{g(t)}\left(1-c_Lc_R+\frac{\omega_L}{\omega_R}s_Ls_R\right), \\
        \nonumber
    e(t) & = \frac{1}{g(t)}\left(1-c_Lc_R+\frac{\omega_R}{\omega_L}s_Ls_R\right).
\end{align}
Then, the following decomposition
\begin{align}
    \nonumber
     &\, e^{-i\frac{t}{2}(\frac12 p^2+\omega_L^2 \frac12 x^2)}
        e^{-i\frac{t}{2}(\frac12 p^2+\omega_R^2 \frac12 x^2)}  \\
    =&\, e^{-if(t)\frac12 x^2}\ e^{-i g(t) \frac12 p^2}\
        e^{-ie(t)\frac12 x^2}
\end{align}
is satisfied for $0\leq t < t^*$, where $t^*$ is the smallest positive root of~$g(t)$.
\end{theorem}
The proof is similar to the previous one.

\section{The Hermite-Fourier methods}
With the presented exact decompositions at hand,
we now solve the discretized GPE (\ref{GP1ODE}) by splitting
methods using the symmetric compositions (\ref{LeapFrogA}) and
(\ref{LeapFrogB}). Let us first consider the case $w=1$ and take the
HO split $A=A_1+B_1$
\begin{align}
    \label{LeapFrogA1}
        \psi_{h,A}^{[2]} & =   e^{-ih (A_1+B_1)/2} \, e^{-ih B} \,  e^{-ih (A_1+B_1)/2}
\\  \label{LeapFrogB1}
        \psi_{h,B}^{[2]} & =  e^{-ih B/2} \, e^{-ih (A_1+B_1)} \, e^{-ih B/2}.
\end{align}
Replacing the exponentials $e^{-ih(A_1+B_1)}$ by \eqref{ABA} or \eqref{BAB},
we obtain four different methods whose computational costs differ considerably.
At first sight, using the FSAL property,
both \eqref{LeapFrogA1} and \eqref{LeapFrogB1} are equivalent from
the computational point of view and require one exponential of $B$
and another one of $A_1+B_1$ per step.
However, a significant difference arises when we plug in the decompositions
\eqref{ABA} or \eqref{BAB}. Only the combination \eqref{LeapFrogB1} with \eqref{BAB}
yields a method that involves only a single FFT call per step
\begin{align}
  \psi_h^{[2]}
   & =   e^{-ih B/2} \, e^{-ih (A_1+B_1)} \, e^{-ih B/2}
   \nonumber \\
   & =   e^{-ih B/2} \,  e^{-i f(h) B_1} \, e^{-ig(h) A_1}
     \, e^{-if(h) B_1} \, e^{-ih B/2}
   \nonumber \\
   & =   e^{-i(h B/2 + f(h) B_1)} \, e^{-ig(h) A_1} \, e^{-i(h B/2+f(h) B_1)}
   \label{leapfrogHF}
\end{align}
and solves exactly the harmonic oscillator for $|h|< h^*$. For any
other combination, more kinetic terms have to be computed per step
since the FSAL property cannot be exploited to full extent and
hence result in more costly \cite{remark3} methods for the same
accuracy.

The general composition (\ref{prod-gen}) can be
rewritten in the same way by replacing each flow $e^{-ia_i A}$ in
(\ref{prod-gen}) by the composition (\ref{BAB}) 
\begin{align}  \label{prod-genHarmOsc} \nonumber
  \Phi_h  \equiv &\;   e^{-i(h b_m B+\alpha_mB_1)} \, e^{-ig(a_mh) A_1} \,
                           e^{-i(h b_m B+\alpha_{m-1}B_1)}\\
                 &\,\cdots e^{-i(h b_1 B+\alpha_1B_1)}%
                        \, e^{-ig(a_1h) A_1}\,e^{-i\alpha_0B_1}
\end{align}
where $\alpha_k=f(a_{k+1}h)+f(a_kh), \ k=0,1,\ldots,m+1$ with
$a_0=a_{m+1}=0$. This method is valid for $|a_i h|< h^*, \ \
i=1,\ldots,m$ and requires only $m$ calls of the FFT and its
inverse, but reaches the same accuracy as if the Hermite functions
were used.

In the more general case with a time-dependent frequency, $\omega(t)$,
starting from the HO splitting, the time-dependent part is first approximated by Magnus expansions \eqref{HO(t)sol} or \eqref{CommutatorFree}.
Theorems \ref{theorem1} and \ref{theorem2} then provide decompositions
to write the product of exponentials in a similar way as in \eqref{prod-genHarmOsc}
but now $\alpha_k=f(a_{k+1}h)+e(a_kh)$ and it is valid for $|a_i h|< t^*, \ \ i=1,\ldots,m$ where $t^*$
is the first zero of $g(t)$. At each stage, one has to compute
$f(a_{i}h),g(a_ih),e(a_ih)$ from $\omega(t)$.

As in the previous case, it only requires $m$ calls of the FFT and
its inverse, like the standard Fourier pseudospectral methods. For
stability reasons, it seems convenient to look for splitting
method whose value of $\max_i\{|a_i| \}$ is as small as possible.

            \section{Numerical examples}

We analyze first the performance of the methods considered in this
work for the one-dimensional problem (\ref{GP1}) with
$\mu=\omega^2=1$, and the pure harmonic trap, i.e. $V_I=0$.

To illustrate the validity of the decomposition presented in
Lemma~\ref{Lemma2}, we first consider the linear problem
($\sigma=0$). We take the ground state at $t=0$ as initial condition
whose exact solution is given by
\[
 \psi(x,t)=\frac{1}{\pi^{1/4}\sqrt{2^nn!}} e^{-it/2} e^{-x^2/2}.
\]
We discretize on the interval $[-10,10]$, to ensure the wave
function and its first derivatives vanish up to round off at the boundaries, and
sample it at $N=1024$ equidistant grid points. We integrate, with only
one time step, from
$t=0$ to $T$ for
$T\in[-\pi,\pi]$,
i.e. forward and backward in time.
We measure the {integrated} error in the wave function,
$\|u_{ex}(T)-u_{a}(T)\|_2$, where $u_{a}(T)$ denotes the
approximate numerical solution obtained using the split
(\ref{BAB}) and $u_{ex}(T)$ is the exact solution at the
discretized mesh. {The result of this comparison is illustrated in
Fig.~\ref{fig1} (left).}  The split (\ref{BAB})
{reproduces}, for $|T|<\pi$, the exact solution up to round off,
as expected.
The right panel in Fig.~\ref{fig1} displays a zoom near a singularity
where the error grows {rapidly due to double precision arithmetic.}

\begin{figure}[ht]%
\begin{center}%
    \hspace*{-0.13cm}
    \psfrag{pisymbol}[][c]{$\pi$}
    \psfrag{mpisymbol}[][c]{$-\pi$}

    \includegraphics[width=\figsize]{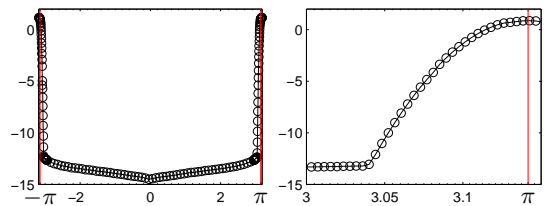}
    \caption{\label{fig1} Error in logarithmic scale for the integration
    of the ground state of the Harmonic potential using the split (\ref{BAB})
    for $T\in[-\pi,\pi]$ (integration forward and backward in
    time). The left and right panels show the 2-norm error
    and a zoom about $T=\pi$, respectively.
    }
\end{center}
\end{figure}

We analyze how the approximation properties of the Hermite
decomposition \eqref{HermiteM} strongly depend on the function in
question and on the chosen number of basis functions, $M$.
We compute the $M$ required to reach round-off precision for the evolution
of a displaced ground state as initial condition, $\psi_{\delta}(x,0)=
e^{-(x-\delta)^2/2}/(\pi^{1/4}\sqrt{2^nn!})$ from $t=0$ to $T=10$ in one time step.
From initial conditions computed on a mesh, this can be accomplished as
follows \cite{perez03amc}:
\beq\label{Hermite2}
    u_{ex}(T) = e^{-iT(A_1+B_1)} u_0 \sim K^T e^{-iTD_1} K \, u_0
\eeq
where $D_1=diag\{\frac12,\frac32,\ldots,\frac{2M-1}{2}\}$, $M$ is
the number of basis elements considered, and
$K_{i,j}=h_{i-1}(x_j), \ i=1,\ldots,M, \ j=1,\ldots,N=512$ with
$h_{n}(x)$ given in (\ref{Hermite}), $x\in[-10,10]$.
For $\delta=\frac{1}{10}$, round off accuracy is
achieved with $M=8$ while for $\delta=2$ it is necessary to take
$M=29$.  We {observe} that the Hermite decomposition is very
sensitive to the initial conditions \cite{lubich08fqt}.
The Hermite basis works efficiently as far as the initial conditions
as well as {the evolution thereof} can be accurately approximated
using a few number of basis {elements} and one has to keep in mind that,
for nonlinear problems, the number of basis functions necessary to
reach a given accuracy can vary along the time integration.

Next, we study the following values for the nonlinearity
parameter: $\sigma=10^{-2}, 1, 10^{2}$. The case $\sigma=10^{-2}$
illustrates the performance of the new methods if applied to
problems (linear or nonlinear) which are small perturbations of
the Harmonic potential whereas the values $\sigma=1,10^2$ are
large enough to demonstrate the nonlinearity effects on the
approximation properties of the methods. Physically, $\sigma$ is
proportional to the number of particles in a Bose-Einstein
condensate and to the interaction strength \cite{pet01}.

For all cases, we choose the initial condition $\psi(x,0)=\rho e^{-(x-1)^2/2}$, with $\rho$ a
normalizing constant. We show in Fig.~\ref{fig2} the value of
$|\psi(x,t)|^2$ at the initial and final time.
The spatial interval is adjusted to ensure the wave function vanishes (up to round off) at
the boundaries, here $[-20,20]$ for $\sigma=0.01$ and
$[-30,30]$ for both $\sigma=1$ and $\sigma=100$ where the wave
function moves faster (we only show the interval $x\in[-5,5]$).
One can appreciate that for strong
nonlinearities the wave function can considerably penetrate the
potential barrier and one expects that an accurate
approximation of these wave functions requires a large number of
Hermite functions when using (\ref{Hermite2}) and hence renders this
procedure inappropriate.

\begin{figure}[ht]%
\begin{center}%
    \hspace*{-0.3cm}\includegraphics[width=.52\textwidth]{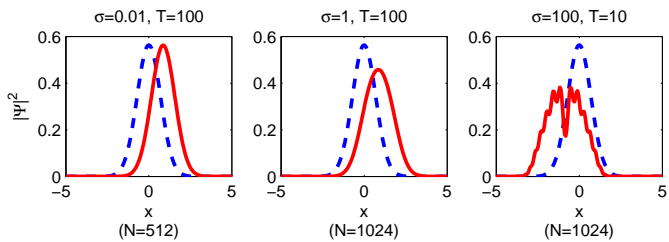}
    \caption{\label{fig2}Exact evolution at $t=T$ (solid line) from the initial
    conditions given by $\psi(x,0)=\rho e^{-(x-1)^2/2}$ (dashed line).
    The number of grid points is given by $N$.}
\end{center}
\end{figure}

\subsection{HO-split versus F-split}

We analyze now the advantages of the HO-split versus
the F-split as given in (\ref{HermiteSplit}) and (\ref{FourierSplit}).

In a first experiment, we fix the symmetric second
order $BAB$ composition (\ref{LeapFrogB1}) and apply it for both splits.
For the HO-split, we compute the harmonic part either
with the decomposition (\ref{BAB}) or in the Hermite basis (\ref{Hermite2})
with different numbers of basis terms.

The parameters, initial conditions and final times are taken as previously for Fig.~\ref{fig2}.
Given a splitting method $X$, we denote by $X_F, X_H$ and $X_HM$ its
implementations with the F-split, the HO-split using the
Hermite-Fourier method and the HO-split using $M$ Hermite basis
functions in (\ref{Hermite2}), respectively.
We measure the error versus the number of exponentials
which can be considered proportional to the computational cost
and plot the results in Fig.~\ref{fig3}.
As expected, HO-splits are advantagous if the system is close to a harmonic oscillator,
i.e. for $\sigma=0.01,1$, and if the initial conditions are accurately approximated by a
few terms in the Hermite expansion.
On the other hand, for strong nonlinearities $\sigma=100$,
the Hermite polynomial based HO split shows a very limited performance, c.f.
the large number of basis terms in the right panel of Fig.~\ref{fig3}.
We stress that if this technique is used for nonlinearities
the number of basis terms should be increased along the time integration and
fixing it bounds the maximally achievable accuracy and its limit
depends on the initial condition and the strength of the nonlinearity.

The Hermite-Fourier method proposed in this work
(using the composition (\ref{BAB})) is clearly superior for weak
perturbations and it keeps similar performance to the F-split for
strong nonlinearities.

\begin{figure*}[ht]
\begin{center}
    \hspace*{0cm}
    \includegraphics[width=\textwidth]{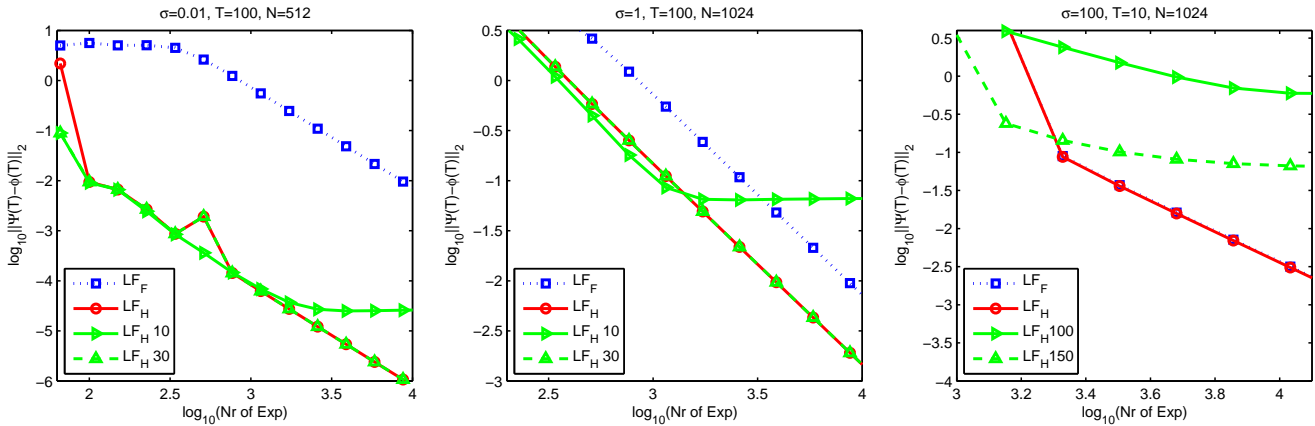}
    \caption{\label{fig3} Error versus the number of exponentials
    in logarithmic scale for different splittings for the
    Leap-Frog method.}
\end{center}
\end{figure*}

Finally, we analyze the performance of different higher order
splitting methods which are useful when high accuracies are
desired. The following methods (whose coefficients are collected
in Table~\ref{tab.1} for the convenience of the reader) are considered: \\
-- $RKN_64$ (the 6-stage fourth-order method from
\cite{blanes02psp}). This is a Partitioned Runge-Kutta-Nystr\"om
method and it is {designed for} the case where $[B,[B,[B,A]]]=0$,
being the case for both the F-split and the HO-split.\\
-- $NI_{4}(8,2)$ (the 4-stage (8,2) $BAB$ method from
\cite{mclachlan95cmi}). This method is addressed to perturbed
systems. One expects a high performance if the contribution from $B$ is small.\\
-- $NI_{5}(8,4)$ (the 5-stage (8,4) $BAB$ method from
\cite{mclachlan95cmi}).

\begin{table}[!ht]
\centering
\caption{Coefficients for several splitting methods.}
\label{tab.1}
 {
\begin{tabular}{ll}
\hline\hline
\multicolumn{2}{c}{The 6-stage 4$th$-order: SRKN$_{6}4$ } \\
\hline\hline
$ b_1= 0.0829844064174052 $ & $a_1=  0.245298957184271$\\
$ b_2= 0.396309801498368  $ & $a_2=  0.604872665711080$\\
$ b_3=-0.0390563049223486 $ & $a_3=  1/2 - (a_1+a_2)$ \\
$ b_4= 1 - 2(b_1+b_2+b_3)$ & $a_4=a_3,\ a_5=a_2, \ a_6=a_1$ \\
$ b_5=b_3,\ b_6=b_2, \ b_7=b_1$
\\
\hline\hline
\multicolumn{2}{c}{The 4-stage (8,2) method: NI$_{4}$ } \\
\hline\hline
$ b_1= 1/20 $   & $a_1= 1/2 - \sqrt{3/28} $\\
$ b_2= 49/18 $  & $a_2= 1/2 - a_1 $\\
$ b_3=  1-2(b_1+b_2)$ & $a_3=a_2, \ a_4=a_1$ \\
$ b_4=b_2, \ b_5=b_1$
\\
\hline\hline
\multicolumn{2}{c}{The 5-stage (8,4) method: NI$_{5}$ } \\
\hline\hline
$ b_1= 0.811862738544516
                        $ & $a_1= -0.00758691311877447
                                                    $\\
$ b_2= -0.677480399532169
                        $ & $a_2= 0.317218277973169
                                                    $\\
$ b_3=  1/2 - (b_1+b_2)$ & $a_3=  1 - 2(a_1+a_2)$ \\
$ b_4=b_3,\ b_5=b_2, \ b_6=b_1 $ & $a_4=a_2,\ a_5=a_1$
\\ \hline\hline
\end{tabular}
}
\end{table}

We analyze in Figures~\ref{fig4}-\ref{fig6} the three problems
specified in Fig.~\ref{fig3}. In the upper panels, the Leap-Frog
methods, LF, are compared with the second order $NI_{4}(8,2)$
methods. In the lower panels we compare the $RKN_64$ methods
against the (8,4) methods jointly with the best among the previous
second order methods.

\begin{figure}[!ht]
\begin{center}%
    \includegraphics[width=\figsize]{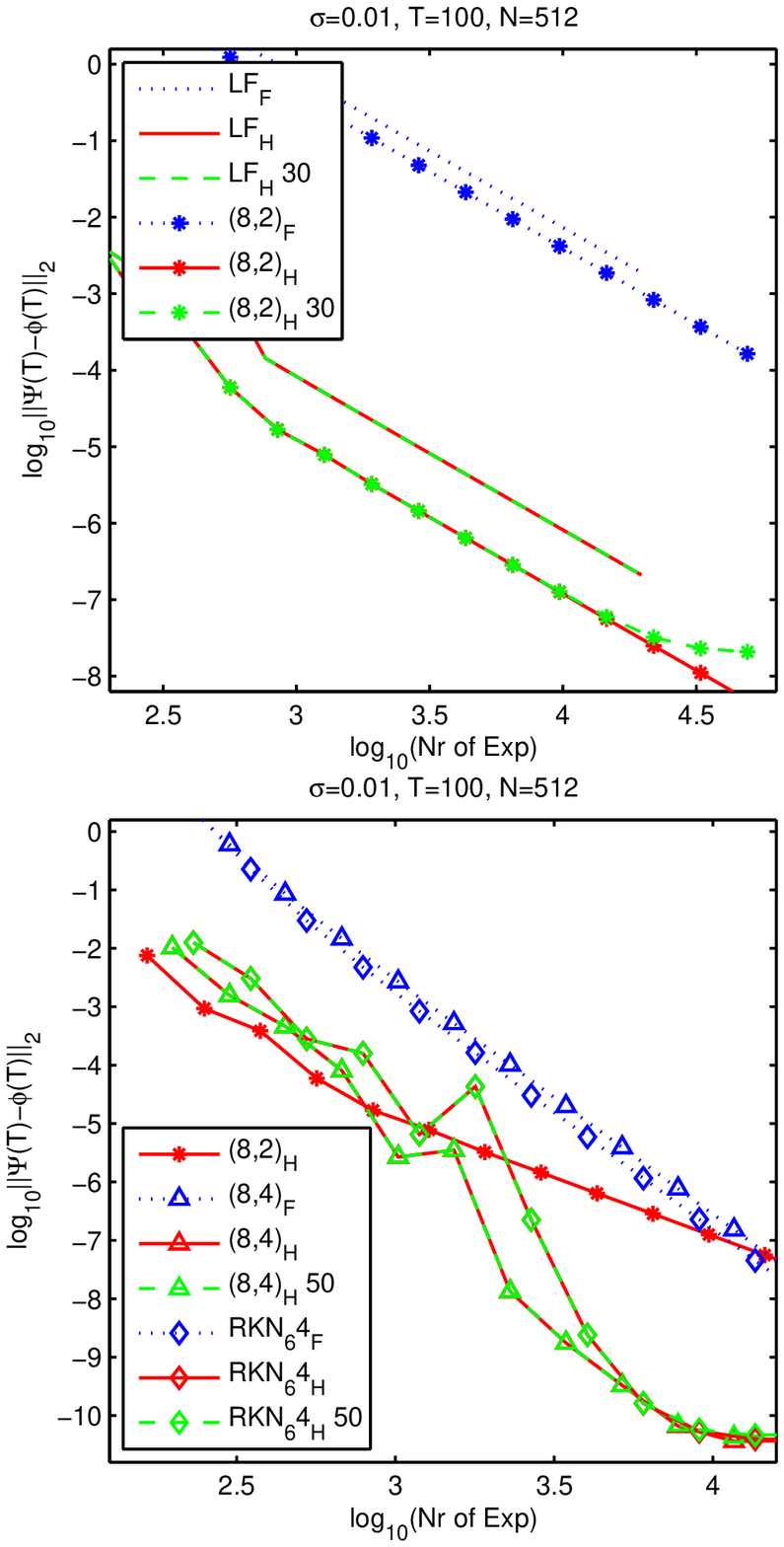}
    \caption{\label{fig4} Comparison of second order (upper panel) and fourth order (lower panel) methods
                for the different splittings and decompositions discussed in the text
                and $\sigma=0.01$.   }
\end{center}
\end{figure}

For a weak nonlinearity, when the system can be considered as a
perturbed harmonic oscillator, we clearly observe that the HO-split is superior to the F-split. In this case,
with a relatively small number of Hermite functions, it is possible
to approach accurately the solution, but this procedure has a
limited accuracy which can deteriorate along the time integration and
depends on the initial conditions. In addition, the methods addressed to
perturbed problems show the best performance: The $(8,2)_H$ method
performs best among the compared when a relatively low accuracy is
desired and the $(8,4)_H$ method takes its place for higher accuracies.

\begin{figure}[!ht]
\begin{center}%
    \includegraphics[width=\figsize]{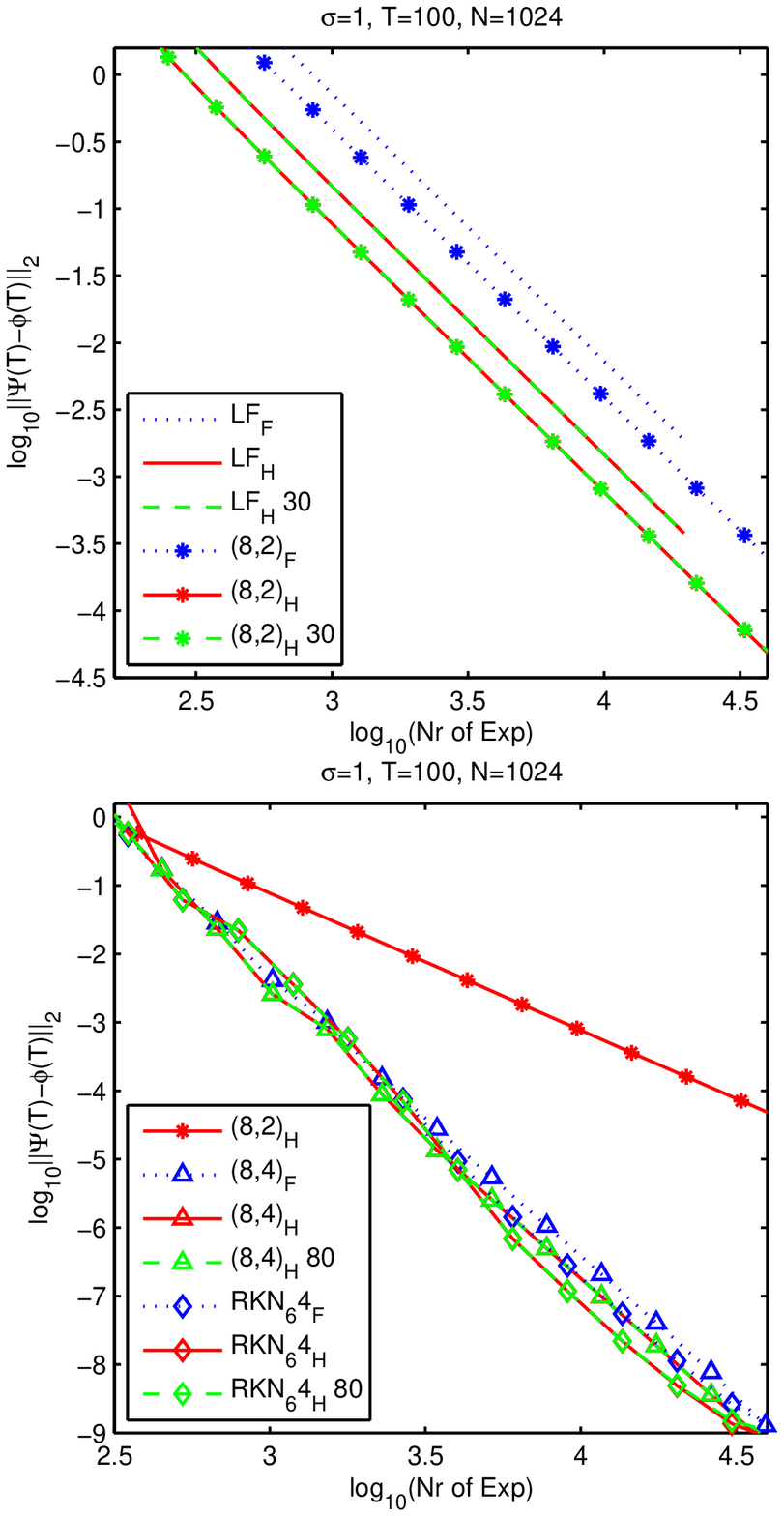}
    \caption{\label{fig5} Same as Fig.~\ref{fig4} for
    $\sigma=1$.    }
\end{center}
\end{figure}

Figure~\ref{fig5} shows the results for $\sigma=1$. It is
qualitatively similar to the previous case yet the HO-split does
not outperform the plain F-split \eqref{FourierSplit} as
significantly as before. Nevertheless, it is important to observe
that, again, the best result is obtained for the HO-split.
Notice that a higher number of Hermite basis functions is
necessary to achieve the same accuracy as the Hermite-Fourier
decomposition.

\begin{figure}[!ht]
\begin{center}%
    \includegraphics[width=\figsize]{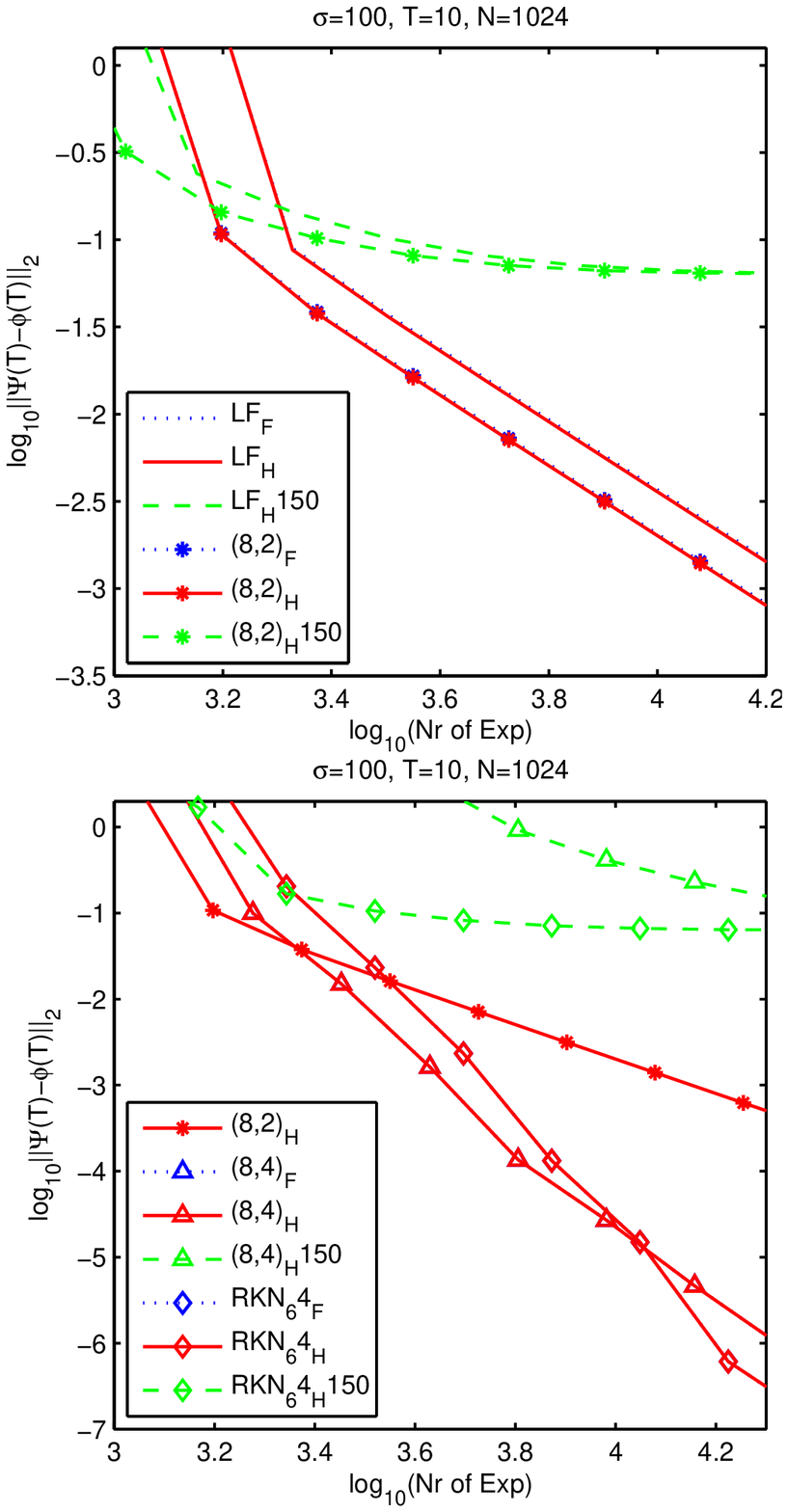}
    \caption{\label{fig6} Same as Fig.~\ref{fig4} for
    $\sigma=100$. F-splittings overlap the
            corresponding (red) Fourier-Hermite curves.
    }
\end{center}
\end{figure}

Figure~\ref{fig6} shows the results for $\sigma=100$.
The HO-split cannot be expected to be particularly useful because
the system is far from being a harmonic oscillator. From
Fig.~\ref{fig2}, we expect a great number of Hermite basis
functions to be required for a sufficiently accurate expansion.
The results in Fig.~\ref{fig6} demonstrate this rather intuitive
expectation, i.e. almost negligible precision despite the large
number of basis terms $M=150$. Remarkably, the proposed HO
decomposition does not show these limitations and reaches the
precision of the F-split \eqref{FourierSplit} because we are
solving the harmonic potential exactly up to spectral accuracy.
For this problem, we observe that the $(8,2)$ method
has the best performance when a relatively low accuracy is
desired, the $(8,4)$ method shows the best performance for medium
accuracies and the $RKN_64$ is the method of choice for higher
accuracies.

\subsection{Time-dependent harmonic oscillator perturbed
by weak quartic anharmonicity} \label{subsec:num:timedep}

We consider now a harmonic oscillator with time-dependent
frequency and perturbed by a weak static quartic anharmonicity
\beq  \label{PertHO}
 i \frac{\partial}{\partial t} \psi =
 \left(\frac{1}{2} p^2 +  \frac{1}{2} \omega^2(t)x^2 \right)\psi +
  \varepsilon_Q \frac14 x^4 \psi .
\eeq
We first consider the case $ \omega^2(t) = A(1+\epsilon \cos(w t))$
with $w=1/2,\ A=4 ,\ \epsilon = 0.1,\ \varepsilon_Q=0.01$.
As reference, we take a highly accurate numerical approximation
as exact solution
 and restrict the spatial domain to $[-20,20]$ for all 
 experiments in this subsection
. We compare the Hermite-Fourier
method with the plain Fourier split, since Hermite polynomials are
not appropriate in a time-dependent setting. For fast oscillating
systems and if high accuracy is needed, the two-exponential
fourth-order approximation of the harmonic oscillator
\eqref{CommutatorFree} can be improved by taking, for example, a
higher order Magnus expansion \eqref{HO(t)sol}. As we have seen,
the solution of \beq  \label{HO}
 i U' = \left(\frac{1}{2} p^2 +  \frac{1}{2} \omega^2(t)x^2
 \right)U
\eeq
 can be written as
\begin{equation}\label{opHO}
 U(t,0) = e^{-i\frac{t}{2}\left(\alpha x^2+\beta (xp+px)+\gamma
 p^2\right)},
\end{equation}
and we have considered, for example, a sixth-order Magnus
integrator \cite{blanes09tme} to approximate the evolution
operator for one fractional time step, $a_ih$, i.e. $U(t,t+a_ih)$.
This is equivalent to take in (\ref{opHO}) $t=a_ih$ and the
parameters $\alpha,\beta,\gamma$ are given by:
\[
\begin{array}{l}
\alpha=\frac{1}{18}(5\omega_1 + 8\omega_2 + 5\omega_3) \\
\quad  +\frac{(a_ih)^2}{486} (\frac{17}{4}(\omega_1^2+\omega_3^2)
+ 8\omega_2^2 + \omega_1\omega_2 +
\omega_2\omega_3 - \frac{37}{2} \omega_1\omega_3), \\
\beta=a_ih\sqrt{\frac{5}{3}}(\omega_3 - \omega_1)(\frac{1}{12} +
\frac{(a_ih)^2}{3240}(5\omega_1 + 8\omega_2 +5\omega_3)),  \\
\gamma=1 + \frac{(a_ih)^2}{54}(\omega_1 - 2\omega_2+ \omega_3)
\end{array}
\]
with $\omega_i=\omega(t_n+c_ih), \ i=1,2,3$ and   $c_1 = 1/2 -
\sqrt{15}/10$, $c_2 = 1/2$, $c_3 = 1/2 +\sqrt{15}/10$,
corresponding to a sixth-order Gaussian quadrature rule. The
obtained operator is then decomposed according to Theorem
\ref{theorem1}.

The results are given in Fig.~\ref{fig:timedep-cos} and corroborate
the superiority of the HO split.

Another interesting example is given by an intense short pulse modeled via
\[
  \omega(t) = w_0\left(1 + \frac{A t}{\cosh^2(B (t-2))}\right).
\]
Varying the parameters $A$ and $B$, the pulse can be sharpened
while keeping its time-average and hence its strength relative to the anharmonicity constant.
Figure~\ref{fig:timedep-pulse} shows the results obtained for a relatively
slow variation of the harmonic potential, for the parameters:
$w_0=4,\ A=0.25,\ B=2$.
Again, the advantageousness of the presented decomposition can be appreciated.
It is already noticeable, that the error introduced by
the time-dependence becomes dominant, and this effect increases for more rapidly varying potentials, e.g. for $B\gg1$. In that case, higher order
approximations of the Magnus expansion are necessary to maintain the benefits
of the Hermite composition.

\begin{figure}[ht]%
\begin{center}%
    \hspace*{-0.13cm}
    \includegraphics[width=\figsize]{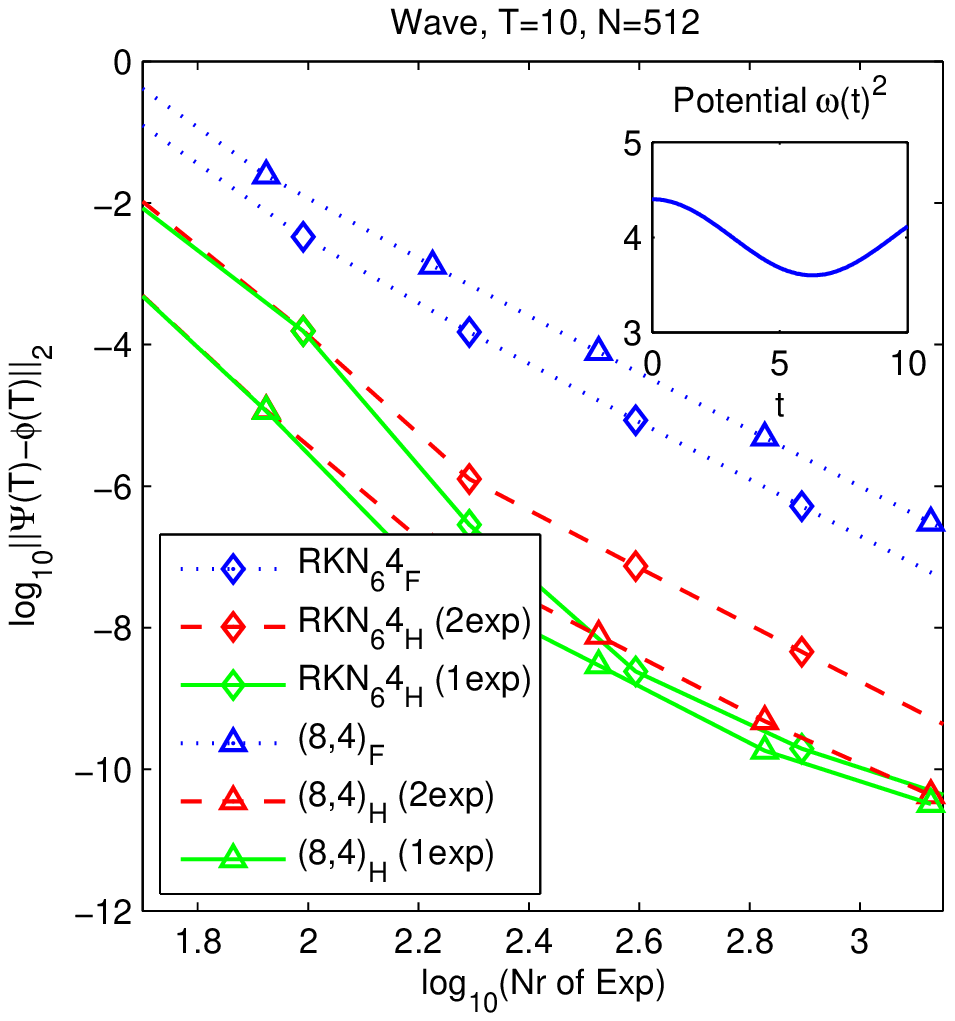}
    \caption{\label{fig:timedep-cos}
        Comparison of Fourier and Fourier-Hermite splittings for two fourth order methods. The (red) dashed line indicates the two-exponential approximation \eqref{CommutatorFree}, the (green) solid line corresponds to the sixth-order Magnus approximation presented in the text \eqref{opHO}.
        The inset shows the evolution of the harmonic trap frequency
        $\omega(t)^2$ and the parameters used for the Hamiltonian are
        $w=1/2,\ A=4 ,\ \epsilon = 0.1,\ \varepsilon_Q=0.01$
        }
\end{center}
\end{figure}

\begin{figure}[ht]%
\begin{center}%
    \hspace*{-0.13cm}
    \includegraphics[width=\figsize]{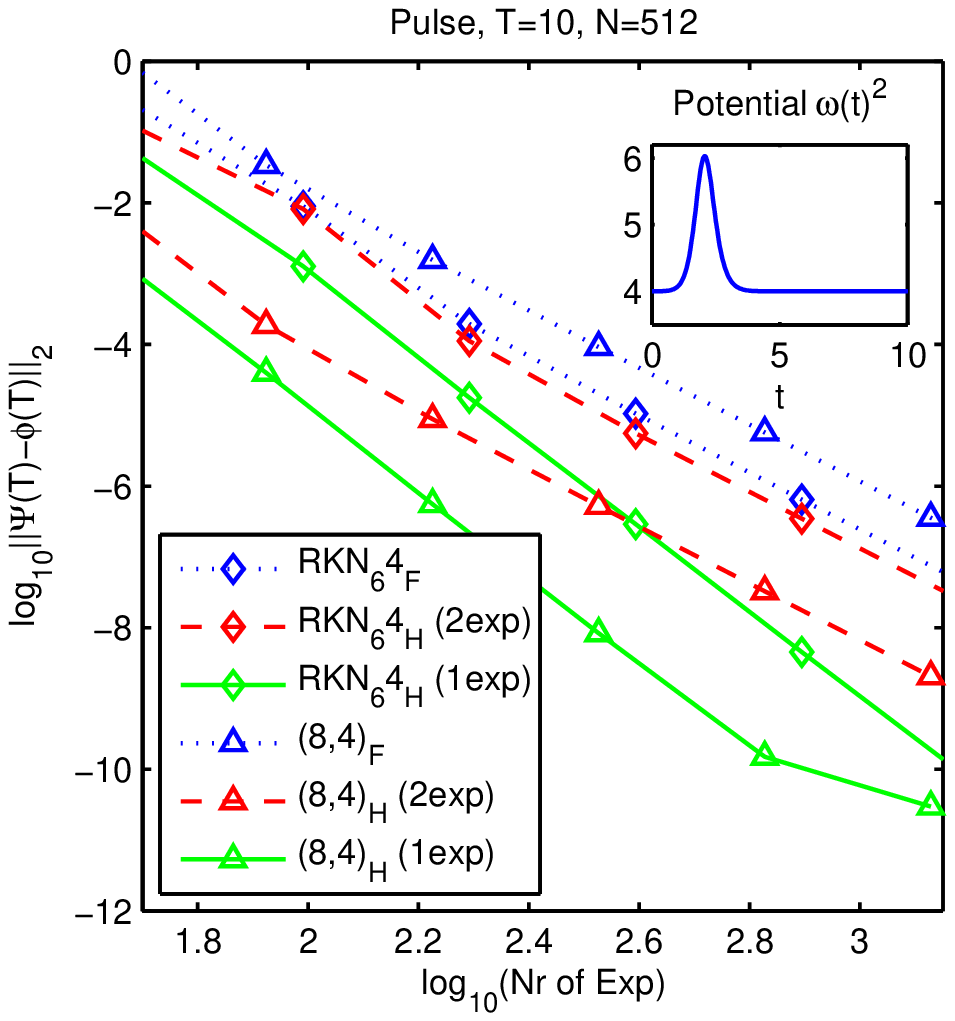}
    \caption{\label{fig:timedep-pulse}
    	Compare Fig.~\ref{fig:timedep-cos}. The parameters used are
        $w_0=4,\ A=0.25,\ B=2$ with a small anharmonicity
        $\varepsilon_Q=0.01$
    }
\end{center}
\end{figure}

\section{Conclusions}
We have presented new Fourier methods for the numerical integration
of perturbations to the time dependent harmonic oscillator which
are useful for both both the Gross-Pitaevskii equation as well as
for the linear Schr\"odinger equation.
Fourier methods have shown a high performance in solving many
different problems which can be split into the kinetic part and a
remainder that is diagonal in the coordinate space.
We have extended the Fourier methods to perturbations of the
time-dependent harmonic potential, and refer to them as Hermite-Fourier
methods. They solve the linear
Schr\"odinger equation with a time-dependent harmonic potential to the
desired order using corresponding Magnus expansions
and up to the accuracy given by the spatial discretization.
These methods are fast to compute since FFTs can be applied
and show a high accuracy when the problem is a small perturbation
of the harmonic potential.
The methods presented in this work extend to perturbed harmonic
potentials in linear quantum mechanics,
c.f. section \ref{subsec:num:timedep},
where it is straightforward to generalize the results to
higher dimensions.

\begin{acknowledgments}%
We would like to acknowledge the referees for their careful
reading and suggestions.
The  authors
acknowledge the support of the Generalitat Valenciana through the
project GV/2009/032. The work of S. Blanes has also been partially
supported by Ministerio de Ciencia e Innovaci\'{o}n (Spain) under
the coordinated project 
MTM2010-18246-C03 (co-financed by the ERDF of the European
Union).
\end{acknowledgments}


\begin{thebibliography}{99}
\bibitem{remarkPhilipp1}Since $F,G,E$ form a basis of the Lie algebra of the
 problem there exist functions $f_i(t)$ that correspond to the summation of
 the Magnus series. The functions can be obtained by solving a set of
 differential equations cf. \cite{wilcox67jmp}
\bibitem{remark3} More precisely, the computational costs are due to a change
of coordinates realized by the Fourier transform which is for this type of
problem equivalent to the number of kinetic terms.
\end{thebibliography}
\end{document}